\documentclass[11pt]{article}
\usepackage{booktabs}
\usepackage{amssymb}
\usepackage{amsmath}
\usepackage{amsthm}
\usepackage{latexsym}

\numberwithin{equation}{section}

\newtheorem{thm}{Theorem}[section]
\newtheorem{lem}[thm]{Lemma}

\newcommand{\be}{\begin{equation}}
\newcommand{\ee}{\end{equation}}
\newcommand{\bea}{\begin{eqnarray}}
\newcommand{\eea}{\end{eqnarray}}
\newcommand{\ba}{\begin{eqnarray*}}
\newcommand{\ea}{\end{eqnarray*}}
\newcommand{\rn}{\mathbb{R}^n}
\newcommand{\Om}{\Omega}
\newcommand{\DK}{\delta_{K+1}}
\newcommand{\x}{\mathbf{x}}
\newcommand{\z}{\mathbf{z}}
\newcommand{\y}{\mathbf{y}}
\newcommand{\bt}{\mathbf{t}}
\newcommand{\rr}{\mathbf{r}}

\begin{document}

\title
 {
 Sparse Signals Recovery from Noisy Measurements by Orthogonal Matching Pursuit
 \thanks
 {
  Research supported in part by the NSF of China under grant 10971189,
 the China Postdoctoral Science Foundation under grant 20100481430
 and Science Foundation of Chinese University under grant 2010QNA3018.
 }
 \thanks{Keywords: Compressed sensing,  restricted isometry property,
 orthogonal matching pursuit, support recovery}
 }
 \author
 {
 Yi Shen
 \thanks{Department of Mathematics and Science, Zhejiang Sci--Tech University, Hangzhou,
 310018, P. R. China}
 and Song Li
 \thanks{Corresponding author: songli@zju.edu.cn, Department of Mathematics, Zhejiang University
 Hangzhou, 310027, P. R. China}
 }
\maketitle

\begin{abstract}
Recently, many practical algorithms have been proposed to  recover
the sparse signal from fewer measurements. Orthogonal matching
pursuit (OMP) is one of the most effective algorithm. In this paper,
we use the restricted isometry property  to analysis the algorithm.
We show that, under certain conditions based on the restricted
isometry property and the signals, OMP will recover the support of
the sparse signal when measurements are corrupted by additive noise.
\end{abstract}

\section{Introduction}
Compressed sensing shows that it is high possibility to reconstruct
sparse signals from their projection onto a small number of random
vectors, possibly corrupted by noise. Let $\|\x\|_0$ denote the
number of nonzero entries of vector $\x$. If $\|\x\|_0<K$, a signal
$\x$ is said to be \emph{$K$-sparse}. Let $A$ be an $m\times n$
measurement matrix with $m<n$. In compressed sensing, we are
interested in recovering the $K$-sparse signal $\x$ from
 \be\label{model}
 \y=A\x+\z,
 \ee
where $\z$ is the noise term. Then, the approach would be to solve
the following $l_0$ minimization problem:
 \be\label{p0}
 \min_{x} \|A\x-\y\|_2 \quad \text{subject to}\quad   \|\x\|_0<K.
 \ee
A greedy algorithm  named orthogonal matching pursuit (OMP) is one
of the efficient approach  to solve (\ref{p0}). The basic idea of
this iterative algorithm is to find the support of the unknown
signal. At each iteration, one column of $A$ that is the most
correlated with the residue is selected. Then the residue is updated
by projecting $y$ onto the linear subspace spanned by the columns
that have been selected. Basic reference for this method are
\cite{DMA,PRK} and \cite{T}.  There are several natural stopping
criteria for OMP \cite{T2}. Let $\rr_k$ be the residual in the each
iteration.
\begin{itemize}
   \item [(1)]Halt after a fixed number of iterations: $k = K$.
   \item [(2)] $l_2$ bounded noise: Halt when no column explains a significant amount of energy in the
        residual: $\|\rr_k\|_2\leq \varepsilon $.
   \item [(3)] $l_\infty$ bounded noise: Halt when no column explains a significant amount of energy
  in the residual: $\| A^*\rr_k \|_{\infty} \leq \varepsilon$ where $A^*$ denotes the transpose of
  $A$.
\end{itemize}

The mutual incoherence property \cite{DE} and the restricted
isometry property \cite{CT} of the measurement matrix have been used
for the analysis of OMP.  Let $A_i$ be the $i$th column of the
matrix $A$. In this paper we assume  $\|A_i\|_2=1$, $i=1,\ldots n$.
The \emph{mutual incoherence} is defined by
$$
\mu(A) = \max_{i\neq j} |\langle A_i, A_j \rangle|,
$$
A given matrix $A$  satisfies the \emph{restricted isometry
property} of order $K$ if there exist a $\delta_K$ such that
 \be\label{RIP}
 (1-\delta_K)\|\x\|_2^2
 \leq
 \|A\x \|_2^2
 \leq
 (1+\delta_K)\|\x\|_2^2 \quad \text{for all}\quad \|\x\|_0\leq K.
 \ee
The smallest constant $\delta_{K}$ is called the \emph{restricted
isometry constant}. Many types of random matrices satisfy the RIP
with high probability, such as subgaussian random matrix \cite{BDDW}
and random partial Fourier matrix \cite{R}. The mutual incoherence
property is stronger than the RIP: $ \DK\leq K\mu(A)$.

In \cite{T}, Tropp has shown
 $
 \mu(A)<\tfrac{1}{2K-1}
 $
is a sufficient condition for reconstructing any $K$-sparse signal
in the noiseless. Then Cai, Wang and Xu proved this condition is
sharp in \cite{CWX}. In \cite{DW}, Davenport and Wakin have showed
that there exist matrices satisfying some RIP but not the mutual
coherence condition via numerical experiments. This motivated them
to establish the RIP-based sufficient conditions. They have proved
that the restricted isometry constant
$\delta_{3K}<\tfrac{1}{3\sqrt{K}}$ is sufficient for OMP to recover
any $K$-sparse signal in $K$ steps. Several papers have improved the
sufficient condition, such as \cite{HZ}  and \cite{LT}. Very
recently, Mo and Shen have improved the sufficient condition to
 \be\label{es}
\DK<\frac{1}{\sqrt{K}+1}.
 \ee
For any $K\geq 2$, they also constructed a matrix with the
restricted isometry constant $ \delta_{K+1} = \tfrac{1}{\sqrt{K}} $
such that OMP can not recover some $K$-sparse signal $\x$ in $K$
iterations. Hence, the estimate (\ref{es}) is near-optimal.

For the noise case,  Cai and Wang  have  provided coherence-based
guarantees for OMP \cite{CW}. This subject was also considered in
\cite{HEE} and \cite{Z}. However, there are few results on the
general model (\ref{model}) by using the RIP. Following the line of
\cite{MS}, we investigate the OMP in the noise case under the
RIP-based conditions.

The rest of paper is organized as follows. In section \ref{set2}, we
shall introduce some notations and investigate some properties of
the restricted isometry constants. In section 3, the main results
are established for OMP recovering the sparse signals with noise.

\section{Preliminaries}\label{set2}
Before going further, we introduce some notations. Suppose $T$ is a
subset of $ \{ 1,\ldots, n \}$. Let $T^c=\{1,2,\ldots,n\}\setminus
T$. For a given matrix $A$,  denote
$$
 A_{T} =\begin{cases}
 A_i,& i\in T,\\
 0,& \text{otherwise.}\end{cases}
$$
For convenience, $A_{T}$ also denotes the submatrix of $A$
corresponding to $T$. We use the same way to define $\x_{T}$ for the
vector $\x\in\rn$. Thus, we have
$$
A_{T}\x = A\x_{T} = A_{T}\x_{T}.
$$
The pseudo inverse of a tall, full-rank matrix $A$ is defined by
$A^{\dag} = (A^*A)^{-1}A$. The support of $\x=\{x_1,\ldots,x_n\}$ is
denoted by $supp(\x)=\{i:x_i\neq 0\}.$  Let $e_i$ be the $i$th
coordinate unit vector in $\rn$. We denote
$$
S_i(\x) := \langle  Ae_i,A\x  \rangle, \quad i=1,\ldots, n,
$$
$$
S_{T}(\x): = \max_{i\in T } |S_i(\x)|,
$$
and
$$
E(\z): = \max_{i\in\{1,\ldots,n\}} |\langle Ae_i, \z\rangle|.
$$
Table \ref{t} shows the framework of OMP.

\begin{table}
\caption{Orthogonal Matching Pursuit}\label{t}\center{}
\begin{tabular}{lll}
 \toprule
 \textbf{Input}: $A$, $\y$\\
 \midrule
 \textbf{Set}: $\Om_0=\emptyset$, $\x_0=0$, $k=1$ \\
 \textbf{while not converge} \\
 $\rr_{k} = \y - A_{\Om_{k-1}} \x_{k-1} $ \\
 $\Om_{k}= \Om_{k-1} \cup \arg\max_{i} |\langle A e_i, \rr_{k}
 \rangle| $ \\
 $\x_k =  (A^*_{\Om_k}A_{\Om_k})^{-1}A_{\Om_k}^* \y $\\
 $k = k+1$\\
 \textbf{end while} \\
 \textbf{set}: $\hat{\x}_{\Om_k} = \x_k$, $\hat{\x}_{\Om_k^C} = 0$ \\
 \midrule
 \textbf{Return}: $\hat{x}$\\
 \bottomrule
\end{tabular}
\end{table}

Now we investigate some properties of the restricted isometry
constant. Lemma \ref{lemtn} were established by Needell and Tropp in
\cite{NT}.
\begin{lem}\label{lemtn}
Let $\x$ be a $K$-sparse vector. Suppose the matrix $A$ has the
restricted isometry constant $\delta_{K}$.  Then for $T\subset
supp(\x)$,
\begin{itemize}
\item [1.]
$
 \|A^{*}_{T}A_{T^c}\x\|_2 \leq
 \delta_{K}\|\x_{T^c}\|_2.
$
\item [2.]
$
 \| (A^*_{T} A_{T})^{-1} x  \|_2
 \leq \frac{1}{1-\delta_K}\|\x\|_2.
$
\end{itemize}
\end{lem}

The following lemma was obtained by Cai and Wang in \cite{CW}.
\begin{lem}\label{lemcw}
Let $\x$ be a $K$-sparse vector with $\Om=supp(\x)$. Suppose that
the matrix $A$ has the restricted isometry constant $\delta_{K}$.
Then for $T\subset \Om$,
\begin{itemize}
\item [1.]
 $
 (1-\delta_{K}) \|\x_{\Om\setminus T}\|_2\leq
 \| A^*_{\Om\setminus
 T} (I-A_T A_T^\dag) A_{\Om\setminus T}\x_{\Om\setminus T} \|_2 \leq
 (1+\delta_{K}) \|\x_{\Om\setminus T}\|_2.
 $
\item [2.]
 $
 (1-\DK)\| \x_{T^c} \|_2 \leq\| A(I-A^{\dag}_{T}A)\x \|_2.
 $
 \end{itemize}
\end{lem}

\begin{lem}\label{lem1}
Let $\x$ be a $K$-sparse vector. Suppose that the matrix $A$ has the
restricted isometry constant $\delta_{K}$. Then for any $ T\subset
\text{supp}(\x)$
 \be\label{lem1eq}
 \|(I-A_{T}^\dag A) \x \|_2
 \leq
 \frac{\|\x_{T^c}\|_2}{1-\delta_K}.
 \ee
\end{lem}
\begin{proof}
Split $\x = \x_{T}+\x_{T^c}$,  we have
 \ba
 (I - A_{T}^\dag A)\x
 &=& \x_{T}+\x_{T^c} - A_{T}^\dag A_{T}\x_{T} - A_{T}^\dag A_{T^c}\x_{T^c} \\
 &=& \x_{T}+\x_{T^c} - \x_{T} - A_{T}^\dag A_{T^c}\x_{T^c} \\
 &=& \x_{T^c}  - A_{T}^\dag
 A_{T^c}\x_{T^c}.
 \ea
By Lemma \ref{lemtn}, we get
 \ba
 \|(I-A_{T}^\dag A) \x \|_2
 &\leq& \| \x_{T^c} \|_2  +  \| A_{T}^\dag
 A_{T^c}\x_{T^c} \|_2 \\
 & \leq &
 \| \x_{T^c} \|_2 + \| (A^*_{T}A_{T})^{-1}A^*_{T}A_{T^c} \x_{T^c} \|_2\\
 & \leq &\| \x_{T^c} \|_2  + \frac{\delta_K}{1-\delta_K}\| \x_{T^c} \|_2 \\
 & \leq &\frac{\| \x_{T^c} \|_2 }{1-\delta_K}.
 \ea
\end{proof}


\section{$l_2$ Bounded Noise}
In this section, we shall prove the main results of the paper. Both
the stopping rule 2 and the stopping rule 3 of OMP for the noise
case are considered. We first consider the noise $\z$ is bounded by
$\|\z\|_2\leq B_2$. Then the stopping rule is $\|\rr_k\|_2\leq B_2$.

\begin{lem}
Suppose $\DK<\tfrac{1}{\sqrt{K}+3}$, we have
 \be\label{delta}
 (1-\DK)^2-\DK(1+\sqrt{K})>0.
 \ee
\end{lem}
\begin{proof}
Simple calculate shows that (\ref{delta}) is equal to
$$
\frac{1}{\DK}+{\DK} > \sqrt{K} + 3.
$$
Thus, $\DK<\tfrac{1}{\sqrt{K}+3}$ is the stronger condition.
\end{proof}
The following results is a key tool in this paper.
\begin{thm}\label{main}
Assume $\DK<\tfrac{1}{\sqrt{K}+3}$. For any given $K$-sparse signal
$\x$. Suppose that the measurement matrix $A$ has the restricted
isometry constant $\DK$ satisfying
 \be\label{result1}
 \|\x_{\Om_k^c}\|_2
 >
 \frac{2(1-\DK)E(\z_k)\sqrt{K-k}}{(1-\DK)^2-\DK(1+\sqrt{K-k})}.
 \ee
where $\z_{k} = (I-A_{\Om_k}A^{\dag}_{\Om_k})\z$. Then OMP selects
an index of the support of $\x$ at the $(k+1)$th iteration.
\end{thm}

\begin{proof}
For a given $K$-sparse signal $\x$,  denote the support of $\x$ by
$\Om$.  Consider the $(k+1)$-th iteration,
 \ba
  \rr_{k+1}
  &=& \y - A_{\Om_k} \x_k \\
  &=& A\x+\z - A_{\Om_k} A^{\dag}_{\Om_k}(A\x+\z)\\
  &=& A(I-A^{\dag}_{\Om_k}A)\x + (I-A_{\Om_k}A^{\dag}_{\Om_k})\z.
 \ea
For simplify, let $\bt_k = (I-A^{\dag}_{\Om_k}A)\x$.  Then we get
 \ba
 \langle Ae_i, \rr_{k+1} \rangle &=&
\langle Ae_i, \y - A_{\Om_k} \x_k \rangle \\
&=& \langle Ae_i, A\bt_k + \z_k \rangle\\
&=&S_i(\bt_k)+ \langle Ae_i , \z_k\rangle.
 \ea
Note that the residual $\rr_k$ are orthogonal to all the selected
columns of $A$, so no index is selected twice. Thus, the sufficient
condition for choosing an index from $\Om\setminus\Om_k$ in the
$(k+1)$th iteration is
 \be\label{lemeq1}
S_{\Om\setminus\Om_k}(\bt_k) - E(\z_k)> |S_i(\bt_k) | + E(\z_k)
\quad \text{for all}\quad i\in \Om^c.
 \ee
In the rest of the proof, we shall give a sufficient condition for
(\ref{lemeq1}) holds.

Note the support of $\bt_k$ is a subset of $\Om$. By Lemma 2.1 in
\cite{C}, we have
 \be\label{lemeq2}
 |S_i(\bt_k)| = |\langle  Ae_i, A\bt_k  \rangle| \leq \delta_{K+1}
 \|\bt_k\|_2\quad
 \text{for all}\  i\in \Om^c.
 \ee
Combine (\ref{lem1eq}) and (\ref{lemeq2}) leads to
 \be\label{lemeq4}
 |S_i(\bt_k)| \leq  \frac{\delta_{K+1}}{1-\delta_{K+1}}\|\x_{\Om_k^c}\|_2,\quad
 \text{for all}\  i\in \Om^c.
 \ee
By Lemma \ref{lemcw}, we obtain
 \be\label{lemeq3}
 {S_{\Om\setminus\Om_k}(\bt_k)}
 \geq
 \frac{\| A^*_{\Om\setminus \Om_k} A \bt_k \|_2}{\sqrt{K-k}}
 \geq
 \frac{ (1-\delta_K)\|\x_{\Om^c_k}\|_2}{\sqrt{K-k}}.
 \ee
It follows from  (\ref{lemeq4}) and (\ref{lemeq3}) that the
sufficient condition for (\ref{lemeq1}) holds is
$$
\frac{ (1-\delta_K)\|\x_{\Om^c_k}\|_2}{\sqrt{K-k}}-
\frac{\delta_{K+1}}{1-\delta_{K+1}}\|\x_{\Om_k^c}\|_2 > 2E(\z_k)
$$
which is simplified to (\ref{result1}).
\end{proof}

\begin{thm}\label{thm2}
  Suppose $\|\z\|_2< B_2$ and $\DK<\frac{1}{\sqrt{K}+3}$. Then OMP with the stopping rule $\|\rr_k\|_2\leq
  B_2$ finds the support of $\x$ if all the nonzero coefficients $x_i$ satisfy
 \be\label{3eq1}
|x_i|
 >
 \frac{2(1-\DK)B_2}{(1-\DK)^2-\DK(1+\sqrt{K})}.
  \ee
\end{thm}

\begin{proof}
We first estimate the $E(\z_k)$. Since $\|\z\|_2\leq B_2$, we have
 $$
 |\langle  Ae_i, \z_k\rangle| \leq
 \|A_i\|_2 \| (I - A_{\Om_k} A^{\dag}_{\Om_k}) \z\|_2\leq
 \|\z\|_2\leq B_2.
 $$
This implies
$$
E(\z_k)\leq B_2.
$$
Hence, (\ref{result1}) is followed by
 \be\label{xeq}
 \|\x_{\Om_k^c}\|_2
 >
\frac{2(1-\DK)B_2\sqrt{K-k}}{(1-\DK)^2-\DK(1+\sqrt{K-k})}.
 \ee
Since
$$
\frac{\|\x_{\Om_k^c}\|_2}{\sqrt{K-k}} \geq \min_{i\in\Om_k^C}|x_i|,
$$
(\ref{xeq})  is implied by
$$
|x_i|
 >
 \frac{2(1-\DK)B_2}{(1-\DK)^2-\DK(1+\sqrt{K})},\quad \text{for all}\
 i \in\Om^c_k.
$$

Now we prove that the OMP do not stop for some $j+1<k$. Consider the
$j+1$ iteration for some $j+1<k$, by Lemma \ref{lemcw}, we have
 \bea\label{thmeq1}
 \| \rr_{j+1} \|_2
 & = & \| \y - A_{\Om_j} \x_j \|_2 \nonumber\\
 & = & \| A(I-A^{\dag}_{\Om_j}A)\x +
 (I-A_{\Om_j}A^{\dag}_{\Om_j})\z\|_2 \nonumber\\
 & \geq &
 \| A(I-A^{\dag}_{\Om_j}A)\x \|_2 - \|(I-A_{\Om_j}A^{\dag}_{\Om_j})\z\|_2 \nonumber\\
 & > &
 (1-\DK)\| \x_{\Om^c_j} \|_2 - B_2.
 \eea
By  (\ref{3eq1}), we obtain
 \be\label{thmeq2}
 (1-\DK)\| \x_{\Om^c_j}\|_2 \geq
 \frac{2(1-\DK)^2B_2}{(1-\DK)^2-\DK(1+\sqrt{K})}\geq 2B_2.
 \ee
It follows from (\ref{thmeq1}) and (\ref{thmeq2}) that $\| \rr_{j+1}
\|_2 > B_2$.
\end{proof}


We assume that $\z$ is zero-mean white Gaussian noise with
covariance $\sigma^2 I_{m\times m}$. Cai, Xu and Zhang have show
that  $z\sim N(0,\sigma^2I_{m\times m})$ satisfies
$$
P(\z\in B_2) \geq 1- 1/m
$$
where $B_2 = \left\{  \z: \|z\|_2 \leq \sigma \sqrt{m+2\sqrt{m\log
m}} \right\}$. With this argument and Theorem \ref{thm2}, we obtain
the following result.

\begin{thm}
Suppose $z\sim N(0,\sigma^2I_{m\times m})$,
$\DK<\frac{1}{\sqrt{K}+3}$ and  nonzero coefficients $x_i$ satisfy
 $$
 |x_i|
 >
 \frac{2(1-\DK)\sigma \sqrt{m+2\sqrt{m\log
 m}}}{(1-\DK)^2-\DK(1+\sqrt{K})}.
 $$
Then OMP with the stopping rule $\|\rr_k\|_2\leq \sigma
\sqrt{m+2\sqrt{m\log m}}$ finds the support of $\x$ with probability
at least $1-1/m$.
\end{thm}

Now we give the RIP-based sufficient conditions for OMP with
$l_\infty$ bounded noise case. Then the stopping rule is $\|
A^*\rr_k\|_\infty\leq B_\infty$.
\begin{thm}
  Suppose $\|A^*\z\|_\infty< B_\infty$ and $\DK<\frac{1}{\sqrt{K}+3}$. Then OMP with the stopping rule
  $\|A^*\rr\|_\infty< B_\infty$ finds the support of $\x$ if all the nonzero coefficients $x_i$
satisfy
 \be\label{4eq1}
|x_i|
 >
 \frac{2(1-\DK)B_{\infty}}{(1-\DK)^2-\DK(1+\sqrt{K})}\left( 1+ \frac{\sqrt{K}}{\sqrt{1-\DK}} \right).
  \ee
\end{thm}
\begin{proof}

Since the proof is similar as the proof of Theorem 4 in  \cite{CW}.
We include a sketch for the completeness. To make sure
(\ref{result1}) of Theorem \ref{main} hold, we first give an
estimation of $E(z_k)$ in the $(k+1)$-th iteration. We have
$$
|\langle Ae_i, \z_k \rangle| \leq |A^*_i \z| + |\langle
Ae_i,A_{\Om_k}A^{\dag}_{\Om_k}\z \rangle | \leq B_{\infty}\left(1+
\frac{\sqrt{k}}{\sqrt{1-\delta_{K+1}}}\right).
$$
Together with (\ref{4eq1}), it implies that (\ref{result1}) holds.
Now consider the $t$-th iteration with $t<k+1$. We obtain
 \ba
 \|A^* z_t\|_{\infty} &\geq&
 \frac{1-\delta_{K+1}}{\sqrt{K-t}}\|\x_{\Om_t}\|_2 -
 \left( 1+ \frac{\sqrt{t}}{\sqrt{1-\DK}}\right)B_{\infty}\\
 &\geq& 2\left( 1+ \frac{\sqrt{K}}{\sqrt{1-\DK}}\right)B_{\infty}
 - \left( 1+ \frac{\sqrt{t}}{\sqrt{1-\DK}}\right)B_{\infty} \\
 &\geq& B_{\infty}.
 \ea
The second inequality is implied by (\ref{4eq1}). Therefore, the OMP
does not stop after $t$-th iteration.

\end{proof}

\end{document}